\documentclass[a4paper,oneside]{scrartcl}

\usepackage[utf8]{inputenc}
\usepackage{amsthm}
\usepackage{amssymb}
\usepackage[english]{babel}
\usepackage{mathtools}
\usepackage{amsfonts}
\usepackage{amsmath}
\usepackage{enumitem}
\usepackage{xcolor}
\allowdisplaybreaks

\newcommand{\C}{\ensuremath{\mathcal{C}}}

\newcommand{\E}{\ensuremath{\mathcal{E}}}
\newcommand{\FF}{\ensuremath{\mathbb{F}}}

\newcommand{\NN}{\ensuremath{\mathbb{N}}}

\renewcommand{\S}{\ensuremath{\mathcal{S}}}

\newcommand{\gauss}[3][q]{\ensuremath{\genfrac[]{0pt}{}{#2}{#3}_{#1}}}

\newcommand{\PG}{\text{PG}}

\newcommand{\erz}[1]{\langle #1 \rangle}

\newcommand{\highlight}[1]{\textit{#1}}
\swapnumbers
\theoremstyle{plain}
\newtheorem{Theorem}{Theorem}[section]
\newtheorem{Lemma}[Theorem]{Lemma}
\newtheorem{Corollary}[Theorem]{Corollary}
\newtheorem{Proposition}[Theorem]{Proposition}

\newtheorem{Remarks}[Theorem]{Remarks}
\newtheorem{Result}[Theorem]{Result}

\theoremstyle{definition}
\newtheorem{Example}[Theorem]{Example}
\newtheorem{Definition}[Theorem]{Definition}

\allowdisplaybreaks

\begin{document}
\title{Maximal co-cliques in the Kneser graph on plane-solid flags in $\PG(6,q)$}
\author{Klaus Metsch\and Daniel Werner\thanks{Justus-Liebig-Universität, Mathematisches Institut, Arndtstraße 2, D-35392 Gießen}}
\maketitle

\setkomafont{disposition}{\normalcolor\sffamily\bfseries}

\begin{abstract}
For $q>27$ we determine the independence number $\alpha(\Gamma)$ of the Kneser graph $\Gamma$ on plane-solid flags in $\PG(6,q)$. More precisely we describe all maximal independent sets of size at least $q^{11}$ and show that every other maximal example has cardinality at most a constant times $q^{10}$.
\end{abstract}

\begin{section}{Introduction}
For integers $n\ge 2$ and prime powers $q$ we denote by $\PG(n,q)$ the $n$-dimensional projective space over the finite field $\FF_q$. A \highlight{flag} $F$ of $\PG(n,q)$ is a set of non-trivial subspaces of $\PG(n,q)$ such that $U\subseteq U'$ or $U'\subseteq U$ for all $U,U'\in F$. Here non-trivial means different from $\emptyset$ and $\PG(n,q)$. The set $\{\dim(U)\mid U\in F\}$ is called the \highlight{type} of the flag $F$. Two flags $F_1$ and $F_2$ of $\PG(n,q)$ are said to be in \highlight{general position}, if for all subspaces $U_1\in F_1$ and $U_2\in F_2$ we have $U_1\cap U_2=\emptyset$ or $\langle U_1,U_2\rangle=\PG(n,q)$.

If $S$ is a non-empty subset of $\{0,1,\dots,n-1\}$, then the \highlight{Kneser graph} of flags of type $S$ is the simple graph whose vertices are the flags of type $S$ of $\PG(n,q)$ with two flags $F$ and $F'$ adjacent if and only if they are in general position. If $S=\{i\}$ for some $i\in\{0,\dots,n-1\}$ then we also speak of the Kneser graph of $i$-dimensional subspaces. Note that this graph, among other generalizations of Kneser graphs, has already been defined by C. Güven \cite{Gueven_Thesis}.

This graph and maximal co-cliques therein were studied in \cite{Blokhuis_Brouwer} for $S=\{1,2\}$ and $n=4$, in \cite{Blokhuis_Brouwer_Gueven} for $S=\{0,n-1\}$ and $n\ge 2$, and in \cite{Blokhuis_Brouwer_Szoenyi} for $S=\{0,2\}$ and $n=4$. The result given in \cite{Blokhuis_Brouwer_Gueven} was already conjectured by Mussche in his thesis \cite{Mussche_Thesis} and is also given by C. Güven (one of the authors of \cite{Blokhuis_Brouwer_Gueven}) in her thesis \cite{Gueven_Thesis}.

In this paper we determine the independence number $\alpha(\Gamma)$ for the Kneser graph $\Gamma$ of type $\{2,3\}$ in $\PG(6,q)$ for $q>27$. We point out that a flag of type $\{2,3\}$ of $\PG(6,q)$ is a self-dual object, hence any independent set of $\Gamma$ can also be seen as an independent set of the Kneser graph of the same type in the dual space of $\PG(6,q)$. To simplify notation, we also denote a flag $\{E,S\}$ of type $\{2,3\}$ by $(E,S)$ where $E$ is a plane and $S$ is a solid. We first provide some examples.

\begin{Example}\label{example_general}
	For a hyperplane $H$ of $\PG(6,q)$ and a maximal set $\E$ of mutually intersecting planes of $H$, we denote by $\Lambda(H,\E)$ the set of all flags $(E,S)$ of type $\{2,3\}$ of $\PG(6,q)$ that satisfy $S\subseteq H$ or $E\in\E$. Dually, for a point $P$ of $\PG(6,q)$ and a maximal set $\S$ of $3$-dimensional subspaces on $P$ any two of which share at least a line, denote by $\Lambda(P,\S)$, the set of all flags $(E,S)$ of type $\{2,3\}$ of $\PG(6,q)$ that satisfy $P\in E$ or $S\in\S$.
\end{Example}

Indeed, the following special case of this example was already covered in a more general setting in \cite{Blokhuis_Brouwer}.

\begin{Example}\label{example_largest}
	For an incident point-hyperplane pair $(P,H)$ of $\PG(6,q)$, let $\Lambda(P,H)$ be the set of all flags $(E,S)$ of type $\{2,3\}$ that satisfy $P\in E$ or $P\in S\subseteq H$ and let $\Lambda(H,P)$ be the set of all flags $(E,S)$ of type $\{2,3\}$ that satisfy $S\subseteq H$ or $P\in E\subseteq H$.\\
	For an incident point-line pair $(P,l)$ of $\PG(6,q)$, let $\Lambda(P,l)$ be the set of all flags $(E,S)$ of type $\{2,3\}$ that satisfy $P\in E$ or $l\subseteq S$.\\
	For an incident pair $(U,H)$ of a $4$-dimensional space $U$ and a hyperplane $H$ of $\PG(6,q)$, let $\Lambda(H,U)$ be the set of all flags $(E,S)$ of type $\{2,3\}$ that satisfy $S\subseteq H$ or $E\subseteq U$.
\end{Example}

We shall show in Proposition \ref{PropositionOnHE} that the sets described in Example \ref{example_general} are maximal independent sets in the Kneser graph of flags of type $\{2,3\}$ in $\PG(6,q)$. Notice that the condition on $\E$ means that $\E$ is an independent set of the Kneser graph of planes of $H\simeq\PG(5,q)$ and that the condition on $\S$ means that $\S$ is an independent set of the Kneser graph of solids on $PG(6,q)$.

The sets constructed in Example \ref{example_largest} are special cases of the ones in Example \ref{example_general} and hence also maximal independent sets. Notice that the first and second examples as well as the third and forth examples in Example \ref{example_largest} are dual to each other.

In order to state our first theorem, we need the Gaussian coefficients \gauss{n}{k} which will be defined in Section 2.

\begin{Theorem}\label{main1}
	For $q>27$, the independence number of the Kneser graph of flags of type $\{2,3\}$ of $\PG(6,q)$ is
	\begin{align*}
		\gauss{6}{4}\cdot \gauss{4}{3}+\gauss{5}{3}\cdot q^3
	\end{align*}
	and the independent sets attaining this bound are precisely the four sets described in Example \ref{example_largest}.
\end{Theorem}

The independence number is of order $q^{11}$. As our proof of this theorem is geometric it also provides a stability result for independence sets. Essentially it says that, for large values of prime powers $q$, Example \ref{example_general} describes all maximal independent sets with at least $27q^{10}$ elements. A precise formulation is given in Theorem \ref{main2}.

\begin{Remarks}
\begin{enumerate}
\item
	For a Hilton-Milner type result for the Kneser graph of plane-solid flags in $\PG(6,q)$ one needs to determine the second largest cardinality of a maximal set $\E$ of mutually intersecting planes of $\PG(5,q)$ and use this in Example \ref{example_general}. However, this number is unknown as it is one of the cases that is not covered by the Hilton-Milner type theorem in \cite{BlokhuisBrouwerETC_HiltonMilnerInVectorSpaces}.
\item
	Every upper bound $b$ for the independence number of a graph with $n$ vertices leads to the lower bound $\chi\ge n/b$ for its chromatic number $\chi$. In our situation this shows that the chromatic number of the Kneser graph of plane-solid flags of $\PG(6,q)$, $q>27$, has chromatic number at least $q^4-q^2+2q+1$. On the other hand, if $U$ is a $4$-space, then the sets $\Lambda(P,\emptyset)$ with $P\in U$ are independent sets whose union covers every vertex, so the chromatic number is at most $q^4+q^3+q^2+q+1$. Using independent sets of the form $\Lambda(P,l)$, we show that this upper bound can be improved.
\item
	We keep all estimations in this paper as easy as possible. A more detailed approach, especially in Lemma \ref{A0B3BAD29C13B92BD5BADB679BB91E4B}, shows that Theorem \ref{main1} holds also for $q=27$. This will appear in the Ph.D. thesis of the second author.
\end{enumerate}
\end{Remarks}
\end{section}

\begin{section}{Preliminaries}
Let $q$ be a prime power and $F_q$ the finite field of order $q$. For integers $n$ and $d$ we define the \highlight{Gaussian coefficient}
\begin{align*}
	\gauss{n}{d}:=\begin{cases}
		\prod_{i=1}^{d}\frac{q^{n+1-i}-1}{q^i-1} &\text{if $0\le d\le n$,}\\
		0 &\text{otherwise.}
	\end{cases}
\end{align*}
For $n\ge 0$ this is the number of $d$-dimensional subspaces of an $F_q$-vector space of dimension $n$. If $0\le d\le n$, and if $D$ is a $d$-dimensional subspace of an $n$-dimensional $F_q$ vector space $V$ then $D$ has exactly $q^{d(n-d)}$ complements in $V$. These two facts can be found in Section 3.1 of \cite{Hirschfeld}. We define
\begin{align*}
	s_q(l,k,d,n):=q^{(l+1)(d-k)}\cdot\gauss{n-k-l-1}{d-k}.
\end{align*}
We also set $s_q(k,d,n):=s_q(-1,k,d,n)$, $s_q(k,n):=s_q(-1,d,n)$ and $s_q(n):=s_q(0,n)$ and omit the subscript $q$ in the following.

\begin{Lemma}
	Given two skew subspaces in $\PG(n,q)$ of dimensions $k$ and $l$ respectively and any integer $d$ the number of $d$-subspaces of $\PG(n,q)$ that contain the $k$-subspace and are skew to the $l$-subspace is $s(l,k,d,n)$.
\end{Lemma}
\begin{proof}
	We prove this for the underlying $F_q$-vector space $V$ of dimension $n+1$ and two skew subspaces $K$ and $L$ of dimension $k+1$ and $l+1$ respectively, where we have to count the number of subspaces $D$ of dimension $d+1$ that contain $K$ and are skew to $L$. Every such subspace $D$ gives rise to a subspace $D+L$ of dimension $d+l+2$ of $V$. Going to the factor space $V/(K+L)$, we see that $V$ has $\gauss{n-k-l-1}{d-k}$ subspaces $U$ of dimension $d+l+2$ that contain $K+L$. For such a subspace $U$ we see in the quotient space $U/K$ that $U$ has $q^{(d-k)(l+1)}$ subspaces $D$ of dimension $d+1$ with $U=L+D$.
\end{proof}

\begin{Lemma}\label{B78EBDD44D76A142EC1D2FE7B90B3BAF}
	If $n\ge 5$ and if $\E$ is a set of planes of $\PG(n,q)$ such that any two distinct planes of $\E$ meet in a line, then $|\E|\le s(n-2)$.
\end{Lemma}
\begin{proof}
	If there exists a line contained in all planes of $\E$, then $|\E|\le s(1,2,n)=s(n-2)$. Otherwise
	there exist planes $E_1,E_2,E_3\in\E$ such that $E_1\cap E_3$ and $E_2\cap E_3$ are distinct lines, which implies that $E_3$ is contained in the 3-space $U:=\langle E_1,E_2\rangle$. In this case, for every further plane $E$ of $\E$ at least two of the lines $E\cap E_1$, $E\cap E_2$ and $E\cap E_3$ are distinct, so $E$ is contained in $U$. Thus, in this case, every plane of $\E$ is one of the $s(2,3)$ planes of $U$.
\end{proof}

The following result has been proven in Theorem 1.4 of \cite{BlokhuisBrouwerETC_HiltonMilnerInVectorSpaces}, where it was formulated in its dual version.

\begin{Result}\label{0950C37D657C2988FA3944093CDE8FD6}
	For $q\ge3$ the independence number $\alpha(\Gamma)$ of the Kneser graph $\Gamma$ of type 3 in $\PG(6,q)$ is given by
	\begin{align*}
		\alpha(\Gamma)=s(3,5)=q^8+q^7+2q^6+2q^5+3q^4+2q^3+2q^2+q+1.
	\end{align*}
	For each hyperplane $H$ of $\PG(6,q)$ the set consisting of all solids of $H$ is an independent set of $\Gamma$ with $\alpha(\Gamma)$ vertices. Every other maximal independent set has cardinality at most $q^6+2q^5+3q^4+3q^3+2q^2+q+1$.
\end{Result}
\end{section}

\begin{section}{Sets of flags of type $\{2,3\}$}
In this section we study sets of flags of type $\{2,3\}$ of $\PG(6,q)$. Recall that we also denote a flag $\{E,S\}$ of type $\{2,3\}$ as the ordered pair $(E,S)$ where $E$ is the plane and $S$ the solid of the flag. Note that two distinct such flags $(E,S)$ and $(E',S')$ are adjacent in $\Gamma$ if and only if $E\cap S'=\emptyset=E'\cap S$. Let $\pi_2$ and $\pi_3$ be the maps from the set of all flags of type $\{2,3\}$ to the set of subspaces of $\PG(6,q)$ with $\pi_2(f):=E$ and $\pi_3(f):=S$ for all flags $f=(E,S)$ of type $\{2,3\}$. For any set $C$ of such flags, we define $\pi_i(C):=\{\pi_i(f):f\in C\}$, $i=2,3$.

\begin{Lemma}\label{EisHcapS}
	Let $\Gamma$ be the Kneser graph of flags of type $\{2,3\}$ of $\PG(6,q)$, let $C$ be an independent set of $\Gamma$, let $H$ be a hyperplane and let $P$ be a point of $\PG(6,q)$.
	\begin{enumerate}[label=\roman*)]
	\item
		Let $\E$ be the set whose elements are the planes $E$ of $H$ for which there exists a solid $S$ with $(E,S)\in C$ and $E=H\cap S$. Then $E\cap E'\neq\emptyset$ for all $E,E'\in\E$, that is, $\E$ is an independent set of the Kneser graph of planes of $H$. Hence $|\E|\le s(1,4)$.
	\item
		Let $\S$ be the set whose elements are the solids $S$ for which there exists a flag $(E,S)\in C$ with $P\in S\setminus E$. Then $|S|\le s(1,4)$.\label{EisHcapSii}
	\end{enumerate}
\end{Lemma}
\begin{proof}
	\begin{enumerate}[label=\roman*)]
	\item\label{BF1A8372A6B28D3FF847C04C9BF3F4E1}
		For $E,E'\in\E$ let $(E,S)$ and $(E',S')$ be flags of $C$ with $S\cap H=E$ and $S'\cap H=E'$. Then $S'\cap E=E'\cap E$ and $S\cap E'=E'\cap E$. Since $C$ is independent, it follows that $E\cap E'\not=\emptyset$. Thus $\E$ is an independent set of the Kneser graph of planes of $H$. The main result of \cite{Frankl_Wilson_EKR_VectorSpaces} implies that $|\E|\le s(2,4)=s(1,4)$.
	\item
		This is a special case of the dual statement of \ref{BF1A8372A6B28D3FF847C04C9BF3F4E1}.\qedhere
	\end{enumerate}
\end{proof}

In the following proposition we investigate the sets constructed in Example \ref{example_general} up to duality.

\begin{Proposition}\label{PropositionOnHE}
	Let $H$ be a hyperplane of $\PG(6,q)$ and let $\Gamma$ be the Kneser graph of flags of type $\{2,3\}$ of $\PG(6,q)$.
	\begin{enumerate}[label=\roman*)]
	\item
		$\Lambda(H,\emptyset)$ is an independent set of $\Gamma$.
	\item
		The maximal independent sets of $\Gamma$ that contain $\Lambda(H,\emptyset)$ are the sets $\Lambda(H,\E)$ for maximal independent sets $\E$ of planes of $H$.
	\item\label{A4D056CBFFA1708935F551E64B54A96C}
		For every maximal independent set $\E$ of the Kneser graph of planes of $H$ we have
		\begin{multline*}
			|\Lambda(H,\E)|=s(3,5)\cdot s(3)+|\E|\cdot q^3\\
				\le q^{11}+2q^{10}+5q^9+7q^8+10q^7+11q^6+11q^5+9q^4+7q^3+4q^2+2q+1
		\end{multline*}
		If equality holds, that is if $|\E|=s(1,4)$, then either there exists a point $P$ in $H$ such that $\E$ consists of all planes of $H$ that contain $P$, or there exists a $4$-dimensional subspace of $H$ such that $\E$ consists of all planes of this $4$-dimensional subspace.
	\end{enumerate}
\end{Proposition}
\begin{proof}
	\begin{enumerate}[label=\roman*)]
	\item
		This follows from the fact that every solid of $H$ meets every plane of $H$.
	\item
		If $\E$ is an independent set of planes in the Kneser graph of planes of $H$, then every solid of $H$ meets every plane of $\E$ non-trivially and every two planes of $\E$ meet non-trivially. Therefore $\Lambda(H,\E)$ is an independent set of $\Gamma$. In order to prove the assertion, it therefore suffices to consider an independent set $C$ of $\Gamma$ with $\Lambda(H,\emptyset)\subseteq C$ and to show that $C$ is contained in $\Lambda(H,\E)$ for a set $\E$ of mutually intersecting planes of $H$.

		Let $C$ be an independent set with $\Lambda(H,\emptyset)\subseteq C$. Let $\E$ be the set of all planes $E$ of $H$ such that $C$ contains a flag $(E,S)$ with $E=S\cap H$. Lemma \ref{EisHcapS} shows that the planes of $\E$ are mutually intersecting. It remains to show that $C\subseteq\Lambda(H,\E)$. Suppose on the contrary that there exists a flag $(E,S)\in C$ with $S\not\subseteq H$ and $H\cap S\not=E$. Then $S\cap H$ is a plane and $E\cap H$ is a line of this plane and $H$ contains a solid $S'$ that is skew to the line $E\cap H$. This implies that $S'$ meets the plane $S\cap H$ in a point and therefore $S'$ contains a plane $E'$ with $E'\cap S\cap H=\emptyset$. Then $(E',S')\in\Lambda(H,\emptyset)\subseteq C$ with $S'\cap E=\emptyset=S\cap E'$ and since $C$ is independent this is a contradiction.
	\item
		Since $H$ contains $s(3,5)$ solids all of which contain $s(2,3)=s(3)$ planes, we have $|\Lambda(H,\emptyset)|=s(3,5)\cdot s(3)$. Every plane $E$ of $H$ lies on $s(2,3,6)-s(2,3,5)=q^3$ solids $S$ with $S\cap H=E$. Hence $|\Lambda(H,\E)|=|\Lambda(H,\emptyset)|+|\E|\cdot q^3$. A special case of a result of \cite{Frankl_Wilson_EKR_VectorSpaces} shows that $|\E|\le s(2,4)=s(1,4)$ with equality if and only if all planes of $\E$ contain a common point of $H$ or lie in a common $4$-subspace of $H$.\qedhere
	\end{enumerate}
\end{proof}

\begin{Lemma}\label{F3289245C0ECD376C35648C3B4FAB775}
	Let $\C$ be an independent set of the Kneser graph of type $\{2,3\}$ in $\PG(6,q)$ and let $\xi\in\NN$ be such that every solid of $\PG(6,q)$ occurs in at most $\xi$ flags of $C$. Let $(E,S)$ be an element of $C$. Then there are at most
	\begin{align*}
		s(2)\cdot s(1,4)\cdot\xi=(q^8+2q^7+4q^6+5q^5+6q^4+5q^3+4q^2+2q+1)\cdot\xi
	\end{align*}
	flags $(E',S')\in\C$ with $E'\cap E=\emptyset$ and $S'\cap E\neq\emptyset$.
\end{Lemma}
\begin{proof}
	Since $C$ is independent, every flag $(E',S')\in\C$ with $E'\cap E=\emptyset$ and $S'\cap E\neq\emptyset$ has the property that $S'\cap E$ is a point $P$ with $P\notin E'$. Hence for every such flag there exists a point $P$ in $E$ with $P\in S'\setminus E'$. Since $E$ has $s(2)$ points and since every solid occurs in at most $\xi$ flags of $C$ Lemma \ref{EisHcapS} \ref{EisHcapSii} proves the statement.
\end{proof}
\end{section}

We now proceed to prove our theorem in three steps, where we consider two special cases in the first two steps: In the first step we only consider independent sets $C$ in which no plane or solid occurs in more than $s(1)$ flags of $C$ and in the second step we consider independent sets $C$ in which no plane or solid occurs in more than $s(2)$ flags of $C$.

\begin{section}{The first special case}
In this section we consider an independent set $C$ of the Kneser graph of type $\{2,3\}$ in $\PG(6,q)$ that has the property that every plane and every solid of $\PG(6,q)$ occurs in at most $q+1$ flags of $C$. Our aim is to prove an upper bound for $|C|$. For every point $P$ we denote the set of all flags $(E,S)\in C$ with $P\in E$ by $\Delta_P(C)$.

\begin{Lemma}\label{D0DCC8ED0186F40E225BD9805289E7E5}
	Let $P_1$, $P_2$ and $P_3$ be non-collinear points of $\PG(6,q)$.
	\begin{enumerate}[label=\roman*)]
	\item\label{BE530922662C195360BFCCCDB690D132}
		If
		\begin{align}
			|\Delta_{P_1}(C)|>(q+1)(6q^6+10q^5+17q^4+15q^3+15q^2+9q+5),\label{EQ_9A14742F33B1E7C14895E6636D3AD281}
		\end{align}
		then there are flags $f_i=(E_i,S_i)\in C$ for $i\in\{1,2,3\}$ with $\dim(\langle E_1,E_2,E_3\rangle)\ge5$, $P_2,P_3\notin S_1,S_2,S_3$ as well as $E_i\cap E_j=P_1$ and $P_2,P_3\notin\langle E_i,E_j\rangle$ for all distinct $i,j\in\{1,2,3\}$.
	\item\label{2628B2CFFBCB7361074AC2EE8F596D65}
		If there are flags $f_1$, $f_2$ and $f_3$ with the properties stated in \ref{BE530922662C195360BFCCCDB690D132} and if
		\begin{align*}
			|\Delta_{P_2}(C)|>(q+1)(6q^6+10q^5+17q^4+18q^3+15q^2+9q+5),
		\end{align*}
		then there are flags $f_i'=(E_i',S_i')\in C$ for $i\in\{1,2,3\}$ with $\dim(\langle E_1',E_2',E_3'\rangle)\ge5$, $P_1,P_3\notin S_1',S_2',S_3'$, $\dim(S_i\cap S_j')\le 1$ for all $i,j\in\{1,2,3\}$ as well as $E_i'\cap E_j'=P_2$ and $P_1,P_3\notin\langle E_i',E_j'\rangle$ for all distinct $i,j\in\{1,2,3\}$.
	\end{enumerate}
\end{Lemma}
\begin{proof}
	\begin{enumerate}[label=\roman*)]
	\item
		We frequently make use of the fact that every plane and every solid occurs in at most $q+1$ flags of $C$. We also make use of the following properties:
		\begin{enumerate}[label=$(Q\arabic*)$]
		\item\label{8334953E9F2BDA468490825772537CAB}
			There are $2\cdot s(1,3,6)-s(2,3,6)=2\cdot s(1,4)-s(3)$ solids that contain $P_1$ and a point of $\{P_2,P_3\}$.
		\item\label{DAD2F4EFC5BCD7104DD588605D149796}
			If $E$ is a plane on $P_1$ and $P$ is a point not contained in $E$, then every plane $E'$ on $P_1$ with $E'\cap E\not=P_1$ or $P\in\erz{E,E'}$ meets the solid $\erz{P,E}$ in at least a line and hence there are at most $s(0,1,3)\cdot s(1,2,6)=s(2)\cdot s(4)$ such planes $E'$.
		\item\label{A0B4D1E4427DA63CECB265E615C2CB1B}
			If $E_1$ and $E_2$ are planes with $E_1\cap E_2=P_1$, then there exist less than $s(1,0,2,4)=q^4$ planes in $\erz{E_1,E_2}$ with $E\cap E_1=E\cap E_2=P_1$.
		\end{enumerate}
		
		From \ref{8334953E9F2BDA468490825772537CAB} and the bound in equation (\ref{EQ_9A14742F33B1E7C14895E6636D3AD281}) we see that there exists a flag $(E_1,S_1)$ in $C$ with $P_1\in E_1$ and $P_2,P_3\notin S_1$. According to \ref{8334953E9F2BDA468490825772537CAB} and \ref{DAD2F4EFC5BCD7104DD588605D149796} the number of flags $(E,S)\in\Delta_{P_1}(C)$ for which $E_1\cap E\not=P_1$ or for which $\erz{E_1,E}$ or $S$ contains a point of $\{P_2,P_3\}$ is at most
		\begin{multline*}
			(q+1)(2\cdot s(1,4)-s(3)+2\cdot s(2)\cdot s(4))
				\\=(q+1)(4q^6+6q^5+10q^4+9q^3+9q^2+5q+3)
		\end{multline*}
		which is smaller than the right hand side of equation (\ref{EQ_9A14742F33B1E7C14895E6636D3AD281}). Therefore, we find a flag $(E_2,S_2)\in\Delta_{P_1}(C)$ such that $E_1\cap E_2=P_1$ and neither of the spaces $\erz{E_1,E_2}$ or $S_2$ contains one of the points $P_2$ and $P_3$. Notice that $\dim(\erz{E_1,E_2})=4$, so for the remaining flag $(E_3,S_3)$ we need that $E_3$ is not contained in $\erz{E_1,E_2}$. Using \ref{8334953E9F2BDA468490825772537CAB}, \ref{DAD2F4EFC5BCD7104DD588605D149796} and \ref{A0B4D1E4427DA63CECB265E615C2CB1B} a similar argument shows that at most
		\begin{multline}\label{eqn_klaus}
		(q+1)(2\cdot s(1,4)-s(3)+4\cdot s(2)\cdot s(4)+q^4)
			\\=(q+1)(6q^6+10q^5+17q^4+15q^3+15q^2+9q+5)
		\end{multline}
		flags of $\Delta_{P_1}(C)$ do not satisfy all of the properties we want for the final flag $(E_3,S_3)$. Since this is the right hand side of (\ref{EQ_9A14742F33B1E7C14895E6636D3AD281}) and thus smaller than $|\Delta_{P_1}(C)|$ we find a flag $(E_3,S_3)$ with the desired properties.
	\item
		We can argue analogously to the proof of \ref{BE530922662C195360BFCCCDB690D132}. However, when choosing the flags $(E'_i,S'_i)$ for $i\in\{1,2,3\}$ we additionally have to avoid all flags $(E,S)\in\Delta_{P_2}(C)$ for which $S$ meets one of the solids $S_1$, $S_2$ and $S_3$ in a plane $\pi$ with $P_1\notin\pi$. For $j\in\{1,2,3\}$ each $S_j$ has $q^3$ planes that do not contain $P_1$, so in total there are at most $3q^3$ solids $S$ that must not appear in any of our desired flags $(E'_i,S'_i)$ for $i\in\{1,2,3\}$ and were not considered before. Therefore, it is sufficient to check that the sum of the number in equation (\ref{eqn_klaus}) and the number $3q^3(q+1)$ is the right hand side of equation (\ref{EQ_9A14742F33B1E7C14895E6636D3AD281}) and thus smaller than $|\Delta_{P_2}(C)|$, which is obviously true.\qedhere
	\end{enumerate}
\end{proof}

\begin{Lemma}\label{A0B3BAD29C13B92BD5BADB679BB91E4B}
	Let $P_1$ and $P_2$ be two distinct points of $\PG(6,q)$ and let $E_1$, $E_2$ and $E_3$ be planes such that $E_i\cap E_j=P_1$ and $P_2\notin\langle E_i,E_j\rangle$ for all distinct $i,j\in\{1,2,3\}$. Furthermore, let $\S$ be the set of all solids of $\PG(6,q)$ with $P_2\in S$ and $S\cap E_i\neq\emptyset$ for all $i\in\{1,2,3\}$. Then we have $|\S|\le 3q^6+6q^5+7q^4+4q^3+2q^2+q+1$.
\end{Lemma}
\begin{proof}
	Let $\E$ be the set of all planes that contain $P_2$ but not $P_1$ and that meet all the planes $E_1$, $E_2$ and $E_3$. There are $s(1,3,6)$ solids on $P_2$ that contain $P_1$. If a solid on $P_2$ does not contain $P_1$ nor any plane of $\E$, then it meets all the planes $E_1$, $E_2$ and $E_3$ in unique points (different from $P_1$) and these three intersection points together with $P_2$ span the solid. Hence, there are at most $(q^2+q)^3$ such solids. Finally, each plane of $\E$ lies in at most $s(0,2,3,6)$ solids that do not contain $P_1$, which shows that the number of solids on $P_2$ that meet $E_1$, $E_2$ and $E_3$ is at most
	\begin{align}
		s(1,3,6)+(q^2+q)^3+|\E|\cdot s(0,2,3,6).\label{klaus_eqn2}
	\end{align}
	It remains to determine an upper bound on $|\E|$. We put $U:=\erz{E_1,E_2}$ with $P_2\notin U$ and if $E\in\E$ we know from $P_2\in E$ and $P_1\notin E$ that $E\cap U$ is a line. We show that every point of $E_3\setminus\{P_1\}$ lies on at most $q$ planes of $\E$. To see this, let $Q$ be a point of $E_3\setminus\{P_1\}$ and suppose that $Q$ lies on at least one plane $E$ of $\E$. Since the lines $\erz{P_2,Q}$ and $E\cap U$ of $E$ are distinct, they meet in a unique point $R$, and $E\cap U$ is a line on $R$. Since $P_2$ is not contained in $\erz{E_1,E_3}$ nor in $\erz{E_2,E_3}$ we have $R\notin E_1$ and $R\notin E_2$. This implies that $R$ lies on exactly $q$ lines of $U$ that meet $E_1$ and $E_2$ but do not contain $P_1$. Since every plane of $\E$ on $Q$ is generated by $P_2$ and such a line, we see that $Q$ lies on at most $q$ planes of $\E$. As there are $q^2+q$ choices for $Q$, we find $|\E|\le (q^2+q)q$. Using this upper bound for $|\E|$, the statement follows from equation (\ref{klaus_eqn2}).
\end{proof}

\begin{Lemma}\label{66147D67EDCF9F5B7E400DEC8ACC8B8B}
	Let $P$ be a point and suppose that there are flags $(E_i,S_i)\in\Delta_{P}(C)$, for $i\in\{1,2,3\}$, such that $E_i\cap E_j=P$ for distinct $i,j\in\{1,2,3\}$. Then every point $Q$ with $Q\notin\langle E_i,E_j\rangle$ and $Q\notin S_i$ for all $i,j\in\{1,2,3\}$ satisfies
	\begin{align*}
		|\Delta_{Q}(C)|\le3q^8+12q^7+21q^6+28q^5+26q^4+18q^3+12q^2+8q+4.
	\end{align*}
\end{Lemma}
\begin{proof}
	For $i\in\{1,2,3\}$ exactly $n:=s(0,2,6)-s(3,0,2,6)$ planes on $Q$ meet $S_i$. Since every plane lies in at most $q+1$ flags of $C$, it follows that there exists at most $3n(q+1)$ flags $(E,S)\in \Delta_Q(C)$ such that $E$ has non-empty intersection with at least one of the solids $S_1$, $S_2$ or $S_3$.

	Every other flag $f=(E,S)\in\Delta_Q(C)$ has the property that its solid $S$ meets $E_1$, $E_2$ and $E_3$. Lemma \ref{A0B3BAD29C13B92BD5BADB679BB91E4B} shows that there at most $n':=3q^6+6q^5+7q^4+4q^3+2q^2+q+1$ such solids. Since each solid lies in at most $q+1$ flags of $C$, there are at most $n'(q+1)$ such flags. Therefore $|\Delta_{Q}(C)|\le (3n+n')(q+1)$ proving the desired bound.
\end{proof}

\begin{Lemma}\label{HilfslemmaKlaus}
	Let $S_1$ and $S_2$ be solids of $\PG(6,q)$ with $\dim(S_1\cap S_2)\le 1$ and let $P$ be a point that is not contained in $S_1\cup S_2$. Then the number of planes that contain $P$ and meet $S_1$ and $S_2$ non-trivially is at most $2q^6+2q^5+3q^4+2q^3+2q^2+q+1$.
\end{Lemma}	
\begin{proof}
	We have $d:=\dim(S_{1}\cap S_{2})\in\{0,1\}$. A line through $P$ meets $S_1$ and $S_2$ if and only if it meets one and hence both of the subspaces $U_1:=\langle P,S_2\rangle\cap S_1$ and $U_2:=\langle P,S_1\rangle\cap S_2$, that is, if the line is contained in the subspace $V:=\erz{U_1,P}$. The subspaces $U_1$ and $U_2$ have the same dimension $u$ where $u=1$ if $d=0$ and $u\in\{1,2\}$ when $d=1$. We have $\dim(V)=u+1$.

	A plane on $P$ that meets $V$ only in $P$ is spanned by $P$, a point of $S_1\setminus U_1$ and a point of $S_2\setminus U_2$, so there are $(s(3)-s(u))^2$ such planes. The number of planes on $P$ that meet $V$ in a line is equal to the number $s(0,1,u+1)$ of lines of $V$ on $P$ multiplied with the number $s(1,2,6)-s(1,2,u+1)$ of planes that meet $V$ in a given line.
	Finally there are $s(0,2,u+1)$ planes on $P$ that are contained in $V$. Hence, the total number of planes on $P$ that meet $S_1$ and $S_2$ non-trivially is equal to
	\begin{align*}
	(s(3)-s(u))^2+s(0,1,u+1)(s(1,2,6)-s(1,2,u+1))+s(0,2,u+1)
	\end{align*}
	The larger value occurs for $u=2$ and gives the bound in the lemma.
\end{proof}

\begin{Lemma}\label{AB78485B265AA118CEC0F6A0C158817C}
	Let $P_1$, $P_2$ and $P_3$ be non-collinear points of $\PG(6,q)$. Suppose that for $i\in\{1,2\}$ and $r\in\{1,2,3\}$ there exist flags $f_{i,r}=(E_{i,r},S_{i,r})\in\Delta_{P_i}(C)$ such that
	\begin{enumerate}[label=$\bullet$]
	\item
		$\forall r,s\in\{1,2,3\}:\dim(S_{1,r}\cap S_{2,s})\le1$ and
	\item
		$\forall i\in\{1,2\},\forall\{r,s,t\}=\{1,2,3\}:P_{3-i},P_3\notin\langle E_{i,r},E_{i,s}\rangle\cup S_{i,r}\text{ and }E_{i,r}\cap E_{i,s}=P_i$.
	\end{enumerate}
	Then $|\Delta_{P_3}(C)|\le 24q^7+54q^6+71q^5+67q^4+48q^3+33q^2+22q+11$.
\end{Lemma}
\begin{proof}
	Because $C$ is independent we know that for every $(E,S)\in C$ and every $i\in\{1,2\}$ we have $S\cap E_{i,r}\neq\emptyset$ for all $r\in\{1,2,3\}$ or $E\cap S_{i,r}\neq\emptyset$ for at least one $r\in\{1,2,3\}$. For $i\in\{1,2\}$ Lemma \ref{A0B3BAD29C13B92BD5BADB679BB91E4B} shows that the number of solids of $\PG(6,q)$ that contain $P_3$ and meet $E_{i,1}$, $E_{i,2}$ and $E_{i,3}$ is at most
	\begin{align}
		m:=3q^6+6q^5+7q^4+4q^3+2q^2+q+1.
	\end{align}
	For every flag $(E,S)\in\Delta_{P_3}(C)$ for which $S$ is not such a solid we know that $E$ is a plane that meets $S_{1,r}$ and $S_{2,s}$ for some $r,s\in\{1,2,3\}$. For any choice of $r,s\in\{1,2,3\}$, Lemma \ref{HilfslemmaKlaus} shows that there exist at most
	\begin{align*}
		n:=2q^6+2q^5+3q^4+2q^3+2q^2+q+1
	\end{align*}
	planes on $P_3$ that meet $S_{1,r}$ and $S_{2,s}$. Since every plane and every solid occurs in at most $q+1$ flags of $C$, it follows that $|\Delta_{P_3}(C)|\le (2m+9n)(q+1)$, as claimed.
\end{proof}

\begin{Proposition}\label{FA8D87A7DFE81F19C5AA9B3D4473192B}
	Let $C$ be an independent set of the Kneser graph of type $\{2,3\}$ in $\PG(6,q)$ with $q\ge 7$ that has the property that every plane and every solid of $\PG(6,q)$ is contained in at most $s(1)$ flags of $C$. Then
	\begin{align*}
	|C|\le 24q^{10}+79q^9+155q^8+210q^7+216q^6+187q^5+140q^4+93q^3+51q^2+22q+5.
	\end{align*}
\end{Proposition}
\begin{proof}
	Let $P_1$ and $P_2$ be distinct points of $\PG(6,q)$ such that $|\Delta_{P_1}(C)|,|\Delta_{P_2}(C)|\ge|\Delta_P(C)|$ for all points $P\neq P_1$. If every flag $(E,S)\in C$ satisfies $\langle P_1,P_2\rangle\cap E\neq\emptyset$, then
	\begin{align*}
		|C|
			&\le(s(2,6)-s(1,-1,2,6))\cdot s(1)\\
			&=q^{10}+3q^9+5q^8+7q^7+8q^6+8q^5+7q^4+5q^3+3q^2+2q+1
	\end{align*}
	since there are $s(2,6)-s(1,-1,2,6)$ planes that meet the line $\erz{P_1,P_2}$ and since every plane lies in at most $q+1$ flags of $C$. Therefore, we may assume that $C$ contains a flag $f=(E,S)$ with $\langle P_1,P_2\rangle\cap E=\emptyset$ and thus $\dim(S\cap\langle P_1,P_2\rangle)\le 0$. Every flag $(E',S')\in C$ either satisfies $E'\cap S\neq\emptyset$ or $E'\cap S=\emptyset\neq S'\cap E$. Lemma \ref{F3289245C0ECD376C35648C3B4FAB775} shows that at most
	\begin{align}
		(q^8+2q^7+4q^6+5q^5+6q^4+5q^3+4q^2+2q+1)\cdot s(1)\label{EQ_2B3A5C22BE91C634F52D670F66AF1351}
	\end{align}
	flags $(E',S')$ of $C$ satisfy $E'\cap S=\emptyset\neq S'\cap E$. Before we count all flags $f'=(E',S')$ with $E'\cap S\neq\emptyset$ we note that we either have
	\begin{align}
		|\Delta_P(C)|\le|\Delta_{P_2}(C)|\le6q^7+16q^6+27q^5+35q^4+33q^3+24q^2+14q+5\label{EQ_7637AB5B76D583899A2DAB19EB3F68E6}
	\end{align}
	for all $P\in\PG(6,q)\setminus\langle P_1,P_2\rangle$ or
	\begin{align*}
		|\Delta_{P_1}(C)|\ge|\Delta_{P_2}(C)|>6q^7+16q^6+27q^5+35q^4+33q^3+24q^2+14q+5.
	\end{align*}
	If the second situation occurs, then Lemma \ref{D0DCC8ED0186F40E225BD9805289E7E5} provides flags $f_{i,j}\in C$ for $i\in\{1,2\}$ and $j\in\{1,2,3\}$ required to apply Lemma \ref{AB78485B265AA118CEC0F6A0C158817C} proving
	\begin{align}
		|\Delta_P(C)|\le 24q^7+54q^6+71q^5+67q^4+48q^3+33q^2+22q+11\label{EQ_DCA300F00031A1E0648D6A5DA657C97D}
	\end{align}
	for all $P\in\PG(6,q)\setminus\langle P_1,P_2\rangle$. Since the bound in equation (\ref{EQ_DCA300F00031A1E0648D6A5DA657C97D}) is weaker than the bound given in equation (\ref{EQ_7637AB5B76D583899A2DAB19EB3F68E6}) we know that it also holds in the first case. In particular, equation (\ref{EQ_DCA300F00031A1E0648D6A5DA657C97D}) holds for all $P\in S\setminus(S\cap\langle P_1,P_2\rangle)$. Note that we chose $f$ such that $S\cap\langle P_1,P_2\rangle$ is at most a point. Now, if $\widehat{P}:=S\cap\langle P_1,P_2\rangle\neq\emptyset$, then, since $P_1$ and $P_2$ are distinct, there is an index $i\in\{1,2\}$ such that $\widehat{P}\neq P_i$ and, using the flags $f_{i,1}$, $f_{i,2}$ and $f_{i,3}$, we may apply Lemma \ref{66147D67EDCF9F5B7E400DEC8ACC8B8B} to see that
	\begin{align*}
		|\Delta_{\widehat{P}}(C)|\le 3q^8+12q^7+21q^6+28q^5+26q^4+18q^3+12q^2+8q+4,
	\end{align*}
	which is weaker than the bound in equation (\ref{EQ_DCA300F00031A1E0648D6A5DA657C97D}) for $q\ge7$. Therefore, the number of all flags $(E',S')$ of $C$ with $E'\cap S\neq\emptyset$ is at most
	\begin{multline*}
		(s(3)-1)\cdot(24q^7+54q^6+71q^5+67q^4+48q^3+33q^2+22q+11)\\
			+3q^8+12q^7+21q^6+28q^5+26q^4+18q^3+12q^2+8q+4\\
				=24q^{10}+78q^9+152q^8+204q^7+207q^6+176q^5+129q^4+84q^3+45q^2+19q+4.
	\end{multline*}
	Together with the upper bound in equation (\ref{EQ_2B3A5C22BE91C634F52D670F66AF1351}) for the remaining flags of $C$, this provides the claimed upper bound on the cardinality of $C$.
\end{proof}
\end{section}

\begin{section}{The second special case}

In this section we study independent sets of the Kneser graph of type $\{2,3\}$ in $\PG(6,q)$ with the property that every plane and every solid of $\PG(6,q)$ occurs in at most $q^2+q+1$ flags of $C$.

\begin{Lemma}\label{klaus1}
	Let $E$ be a plane and suppose that the solids $S$ with $(E,S)\in C$ span a subspace $H$ of dimension at least $5$. Suppose also that every plane of $\PG(6,q)$ occurs in at most $s(2)=q^2+q+1$ flags of $C$. Then the number of flags $(E',S')\in C$ with $E'\cap E=\emptyset$ is at most $s(1,4)\cdot s(2)\cdot (s(2)+1)$.
\end{Lemma}
\begin{proof}
	Let $M$ be the set consisting of all flags $(E',S')$ of $C$ such that $S'\cap E=\emptyset$ and let $N$ be the set consisting of all flags $(E',S')$ of $C$ such that $S'\cap E\not=\emptyset$ and $E'\cap E=\emptyset$. Every flag $(E',S')\in C$ with $E'\cap E=\emptyset$ lies in $M\cup N$.
	Lemma \ref{F3289245C0ECD376C35648C3B4FAB775} applied with $\xi=s(2)$ shows that $|N|\le s(1,4)\cdot s(2)^2$. For an upper bound on the number of flags in $M$, we let $\E$ denote the set of all planes that occur in a flag of $M$. The hypothesis of this lemma shows that $|M|\le |\E|\cdot s(2)$. In order to prove the statement, it remains to show that $|\E|\le s(1,4)$.

	Consider $E'\in \E$. Let $S'$ be a solid with $(E',S')\in M$. Then $S'\cap E=E'\cap E=\emptyset$. Since $C$ is independent and since $S'\cap E=\emptyset$, then $E'$ meets every solid $S$ for which $(E,S)\in C$. Then every such solid $S$ is spanned by $E$ and a point of $E'$, so $H\subseteq \erz{E,E'}$. Since $H$ has dimension at least $5$, it follows that $H$ has dimension 5 and that $H=\erz{E,E'}$ for all $E'\in\E$. Lemma \ref{EisHcapS} shows that $|\E|\le s(1,4)$.
\end{proof}

\begin{Proposition}\label{580D3E3B704A762B79A76A6CCC614713}
	Let $C$ be an independent set of the Kneser graph of type $\{2,3\}$ in $\PG(6,q)$ with $q\ge 8$ that has the property that every plane and every solid of $\PG(6,q)$ occurs in at most $s(2)$ flags of $C$. Then
	\begin{align*}
		|C|\le 24q^{10}+79q^9+155q^8+210q^7+218q^6+189q^5+142q^4+95q^3+53q^2+22q+5.	
	\end{align*}
\end{Proposition}
\begin{proof}
	Let $\E$ be the set consisting of all planes of $\PG(6,q)$ that lie in at least $q+2$ flags of $C$, and let $\S$ be the set consisting of all solids of $\PG(6,q)$ that lie in at least $q+2$ flags of $C$. We distinguish three cases.

	Case 1. We assume that $|\E|\le s(4)$ and $|\S|\le s(4)$. In this case we choose a subset $C'$ of $C$ such that every plane and every solid of $C'$ lies in at most $q+1$ flags of $C'$. Since every plane and solid lies in at most $q^2+q+1$ flags of $C$, we can find such a subset with $|C'|\ge |C|-(|\E|+|\S|)q^2$ and then $|C|\le |C'|+2\cdot s(4)\cdot q^2$. Now the statement follows by applying Proposition \ref{FA8D87A7DFE81F19C5AA9B3D4473192B} to $C'$.

	Case 2. We assume that $|\E|>s(4)$. Lemma \ref{B78EBDD44D76A142EC1D2FE7B90B3BAF} proves the existence of planes $E_1,E_2\in\E$ satisfying $\dim(E_1\cap E_2)\le 0$. From Lemma \ref{klaus1} we know that at most
	\begin{align}
		2\cdot s(1,4)\cdot s(2)\cdot (s(2)+1)\label{6362106A47EF61247CAF823DC9C680DE}
	\end{align}
	flags $(E,S)\in C$ satisfy $E\cap E_1=\emptyset$ or $E\cap E_2=\emptyset$. It remains to find an upper bound on the number of flags in $C$ whose planes meet both $E_1$ and $E_2$. Therefore, we count the number of planes of $\PG(6,q)$ that meet $E_1$ and $E_2$. First consider the case that $E_1\cap E_2$ is a point $Q$. In this case there are $s(0,2,6)$ planes on $Q$, there are $(s(2)-1)^2(s(1,2,6)-(2\cdot s(0,1,2)-1))$ planes that do not contain $Q$ and meet both $E_1$ and $E_2$ in exactly one point and there are $2\cdot s(0,-1,1,2)(s(2)-1)$ planes that do not contain $Q$ and meet $E_1$ or $E_2$ in a line and the other plane in a point. Thus, in this case the number of planes that meet $E_1$ and $E_2$ is equal to
	\begin{align*}
		n:=2q^8+4q^7+6q^6+4q^5+4q^4+3q^3+2q^2+q+1.
	\end{align*}
	If $E_1$ and $E_2$ are skew than a similar calculation shows that there are even less than $n$ planes that meet $E_1$ and $E_2$, so that $n$ is an upper bound for the number of planes that meet $E_1$ and $E_2$ in both situations. Since every plane lies in at most $s(2)$ flags of $C$, it follows that there are at most $n\cdot s(2)$ flags $(E,S)\in C$ such that $E$ meets $E_1$ and $E_2$. Together with the count in equation (\ref{6362106A47EF61247CAF823DC9C680DE}) we find $|C|\le n\cdot s(2)+2\cdot s(1,4)\cdot s(2)\cdot (s(2)+1)$ and this bound is better than the one in the statement.

	Case 3. We assume that $|\S|>s(4)$. This is dual to Case 2.
\end{proof}
\end{section}

\begin{section}{Proof of the theorem}

In this section, $\Gamma$ denotes the Kneser graph of plane-solid flags in $\PG(6,q)$ and $C$ denotes a maximal independent set of $\Gamma$.

\begin{Lemma}\label{0398FE63758F2951C1241B8A6382A440_Klaus}
	\begin{enumerate}[label=\roman*)]
	\item
		Every solid $S$ of $\PG(6,q)$ has a subspace $U$ with the following property: For every plane $E$ of $S$ we have $(E,S)\in C$ if and only if $U\subseteq E$.
	\item
		For every plane $E$ of $\PG(6,q)$ there exists a subspace $U$ containing $E$ with the following property: For every solid $S$ on $E$ we have $(E,S)\in C$ if and only if $S\subseteq U$.
	\end{enumerate}
\end{Lemma}
\begin{proof}
	Since the two statements are dual to each other, it suffices to prove the first statement. Thus consider a plane $E$ and let $\S$ be the set of solids $S$ satisfying $E\subseteq S$ and $(E,S)\in C$. In the quotient space $\PG(6,q)/E$ the set $\{S/E\mid S\in\S\}$ is a set of points and we have to show that this set is a subspace of $\PG(6,q)/E$. In that regard, it is sufficient to show for any two distinct solids $S_1,S_2\in \S$ and every solid $S$ with $E\subseteq S\subseteq \erz{S_1,S_2}$ we have $S\in\S$. Let $S$ be such a solid. If $(E',S')$ is any flag of $C$ then either $S'\cap E\not=\emptyset$ or $E'$ meets $S_1\setminus E$ and $S_2\setminus E$. In the second case $E'$ meets $\erz{S_1,S_2}$ in a line and hence $E'$ meets $S$. Thus for every $(E',S')\in C$ we have $E\cap S'\not=\emptyset$ or $E'\cap S\not=\emptyset$. This shows that $C\cup\{(E,S)\}$ is an independent set of $\Gamma$ and since $C$ is a maximal independent set we have $(E,S)\in C$, that is, $S\in\S$.
\end{proof}

\begin{Definition}
	A plane $E$ will be called \highlight{saturated} (for $C$) if $(E,S)\in C$ for all solids $S$ of $\PG(6,q)$ that contain $E$. Dually, a solid $S$ will be called \highlight{saturated} (for $C$), if $(E,S)\in C$ for all planes $E$ of $S$.
\end{Definition}	

\begin{Lemma}\label{saturated_lemma}
	\begin{enumerate}[label=\roman*)]
	\item
		For every saturated solid $S$ and every flag $(E',S')\in C$ we have $E'\cap S\neq\emptyset$.\label{saturated1}
	\item
		If $S$ is a solid with $S\cap E'\neq\emptyset$ for all flags $(E',C')$ of $C$, then $S$ is saturated.\label{saturated2}
	\item
		If $S$ and $S'$ are saturated solids, then $\dim(S\cap S')\ge 1$.\label{saturated3}
	\item
		Let $H$ be a hyperplane of $\PG(6,q)$ and suppose that $E\subseteq H$ for all flags $(E,S)\in C$. Then every solid of $H$ is saturated. \label{saturated4}
	\end{enumerate}
\end{Lemma}
\begin{proof}
	\begin{enumerate}[label=\roman*)]
	\item
		Suppose that there is a flag $(E',S')\in C$ with $E'\cap S=\emptyset$. Since $\PG(6,q)$ has dimension $6$, it follows $S'\cap S$ is a point $P$ with $P\notin E'$. Let $E$ be a plane of $S$ with $P\notin E$. Then $E\cap S'=\emptyset$. As $S$ is a saturated solid we have $(E,S)\in C$. But then $(E,S)$ and $(E',S')$ are flags of the independent set $C$ with $E\cap S'=\emptyset$ and $E'\cap S=\emptyset$, a contradiction.
	\item
		Let $E$ be a plane of $S$. We have to show that $(E,S)\in C$. Since $S\cap E'\not=\emptyset$ for every flag $(E',S)$ of $C$, the set $C\cup\{(E,S)\}$ is independent. Maximality of $C$ implies $(E,S)\in C$.
	\item
		Assume to the contrary that $S$ and $S'$ only meet in a point $P$. Choose planes $E$ of $S$ and $E'$ of $S'$ with $P\notin E,E'$. Then $S\cap E'=\emptyset=S'\cap E$. Hence $(E,S)$ and $(E',S')$ are adjacent elements of the Kneser graph $\Gamma$. As $C$ is independent, this is a contradiction.
	\item
		Let $S$ be a solid of $H$. The dimension formula shows that $S\cap E\not=\emptyset$ for all planes $E$ of $H$. Therefore part \ref{saturated2} shows that $S$ is saturated.\qedhere
	\end{enumerate}
\end{proof}

\begin{Lemma}\label{520AA08BDD632D910FFDEAC535B0337F}
	Let $C$ be a maximal independent set of $\Gamma$. If there are more than $c:=q^7+2q^6+2q^5+3q^4+2q^3+2q^2+q+1$ saturated solids for $C$, then $C=\Lambda(H,\E)$ for some hyperplane $H$ and some maximal independent set $\E$ of the Kneser graph of planes of $H$ (cf. Example \ref{example_general}).
\end{Lemma}
\begin{proof}
	Let $\S$ be the set of saturated solids in $\Pi_3(C)$. We have $c>q^6+2q^5+3q^4+3q^3+2q^2+q+1$ and according to Corollary \ref{saturated_lemma} \ref{saturated3} we have $\dim\left(S_1\cap S_2\right)\ge 1$ for all $S_1,S_2\in\S$. Result \ref{0950C37D657C2988FA3944093CDE8FD6} shows that there exists a hyperplane $H$ containing all saturated solids. If there would be a flag $(E,S)\in C$ such that $E\not\le H$, then according to Lemma \ref{saturated_lemma} \ref{saturated1} all solids of $\S$ would have non-empty intersection with the line $E\cap H$ and thus
	\begin{align*}
	|\S|\le s(3,5)-s(1,-1,3,5)=q^7+2q^6+2q^5+3q^4+2q^3+2q^2+q+1,
	\end{align*}
	which is a contradiction. This shows that $E\subseteq H$ for all planes $E\in\Pi_2(C)$. Lemma \ref{saturated_lemma} \ref{saturated4} shows that all solids of $H$ are saturated. This means that $\Lambda(H,\emptyset)\subseteq C$. Proposition \ref{PropositionOnHE} now proves the statement.
\end{proof}

\begin{Theorem}\label{main2}
	Suppose that $q\ge8$ and that $C$ is a maximal independent set in $\Gamma$ with
	\begin{align*}
		|C|>26q^{10}+83q^9+159q^8+216q^7+222q^6+193q^5+144q^4+97q^3+53q^2+22q+5.	
	\end{align*}
	Then $C=\Lambda(H,\E)$ for a hyperplane $H$ and a maximal set $\E$ of mutually intersecting planes of $H$, or $C=\Lambda(P,\S)$ for a point $P$ and a maximal set $\S$ of solids on $P$ any two of which share at least a line.
\end{Theorem}
\begin{proof}
	The class of examples described in Example \ref{example_general} is closed under duality. In view of Lemma \ref{520AA08BDD632D910FFDEAC535B0337F} we may assume that there exists at most $c:=q^7+2q^6+2q^5+3q^4+2q^3+2q^2+q+1$ saturated solids, and, dually, that there are at most $c$ saturated planes. For every saturated plane $E$ choose one hyperplane $H_E$ on $E$, and for every saturated solid $S$ choose a point $P_S$ of $S$. Let $C'$ be the subset of $C$ that is obtained from $C$ by removing all flags $(E,S)$ such that $E$ is saturated and $S$ is not contained in $H_E$ and by removing all flags $(E,S)$ such that $S$ is saturated and $E$ does not contain $P_S$. Then
	$|C'|\ge |C|-2cq^3$, that is
	\begin{align}\label{EQ_D12997E44FA99026A1CC9471252C3E24}
		|C|\le 2(q^7+2q^6+2q^5+3q^4+2q^3+2q^2+q+1)\cdot q^3+|C'|.
	\end{align}
	Lemma \ref{0398FE63758F2951C1241B8A6382A440_Klaus} shows for every plane $E$ that is not saturated for $C$ has the property that the solids $S$ with $(E,S)\in C$ span a proper subspace of $\PG(6,q)$. Therefore the construction of $C'$ implies that every plane $E$ has the property that the solids $S$ with $(E,S)\in C'$ span a proper subspace of $\PG(6,q)$. Consequently every plane of $\PG(6,q)$ lies in at most $q^2+q+1$ flags of $C'$. Dually, every solid $S$ of $\PG(6,q)$ lies in at most $q^2+q+1$ flags of $C'$. Therefore Proposition \ref{580D3E3B704A762B79A76A6CCC614713} proves an upper bound for $|C'|$. Now (\ref{EQ_D12997E44FA99026A1CC9471252C3E24}) proves the bound for $|C|$ that is given in the statement.
\end{proof}

\begin{Corollary}\label{Cor_main}
	For $q>27$, the maximal independent set in the Kneser graph of flags of type $\{2,3\}$ in $\PG(6,q)$ with $|C|\ge q^{11}+2q^{10}$ are the independent sets described in Example \ref{example_general}.
\end{Corollary}

Theorem \ref{main1} follows from this Corollary and Proposition \ref{PropositionOnHE}~\ref{A4D056CBFFA1708935F551E64B54A96C}.
\end{section}

\begin{section}{Bounds on the chromatic number of $\Gamma$}

Let $\Gamma$ be the Kneser graph of flags of type $\{2,3\}$ in $\PG(6,q)$. The chromatic number of $\Gamma$ is the smallest number $\chi$ such that the vertex set can be represented as the union of $\chi$ independent sets. Using the upper bound $\alpha$ for the size of such an independent set this immediately gives the bound $\chi\ge\frac{n}{\alpha}$. With the upper bound from Theorem \ref{main1} we find

\begin{Proposition}
	For $q>27$, the chromatic number of $\Gamma$ is at least $q^4-q^2+2q+1$.
\end{Proposition}

On the other hand, if $V$ is a subspace of dimension $4$ of $\PG(6,q)$, then the independent sets $\Lambda(P,\emptyset)$ with $P\in V$ comprise all vertices of $\Gamma$, so we have the trivial upper bound $\chi\le s(4)=q^4+q^3+q^2+q+1$. We can slightly improve this bound.

\begin{Proposition}
	The chromatic number $\chi$ of $\Gamma$ satisfies $\chi\le q^4+q^3+q^2+1$.
\end{Proposition}
\begin{proof}
	Consider a point $P$, a line $l$, a plane $E$ and a $4$-space $V$ that are mutually incident. Let $Q$ be a point of $V$ that is not in $E$. Let $l_1,\dots,l_q$ be the lines of plane $\erz{l,Q}$ with $P\in l_i$ and $Q\notin l_i$, let $E_1,\dots,E_q$ be the planes of $\erz{E,Q}$ with $l\subseteq E_i$ and $Q\notin E_i$, and let $S_1,\dots,S_q$ be the solids of $V$ with $E\subseteq S_i$ and $Q\notin S_i$. For $i\in\{1,\dots,q\}$ put $M_i:=l_i\cup (E_i\setminus l)\cup (S_i\setminus E)$. Then $|M_i|=q^3+q^2+q+1$ with $M_i\cap M_j=P$ for distinct $i,j\in\{1,\dots,q\}$ and the union of the sets $M_1,\dots,M_q$ is $\{P\}\cup V\setminus\langle P,Q\rangle$. Let $\{Q_1,\dots,Q_q\}=\langle P,Q\rangle\setminus\{P\}$ and consider the independent set $\Lambda(X,\langle X, Q_i\rangle)$ for $X\in M_i$ and $i\in\{1,\dots,q\}$. Then for $i\in\{1,\dots,q\}$ all lines of $V$ on $Q_i$ occur in one of these sets and every solid that contains $Q_i$ contains a line $\langle X,Q_i\rangle$ with $X\in M_i$. Therefore the union of the sets $\Lambda(X,\langle X,Q_i\rangle)$ for $i\in\{1,\dots,q\}$ covers all vertices of $\Gamma$.
\end{proof}
\end{section}

\end{document}